\documentclass[a4paper,10pt]{amsart}

\usepackage{graphicx}

\usepackage{geometry}
\usepackage{amsthm}
\usepackage{amssymb}
\usepackage{amsmath}
\usepackage{accents}
\usepackage{hyperref}
\usepackage{xcolor}
\usepackage{enumerate}
\usepackage{listings}
\usepackage{complexity}
\usepackage{blkarray}
\usepackage{braket}
\usepackage{lmodern}
\usepackage[english]{babel}
\usepackage{enumitem}

\usepackage[font=small,labelfont=bf]{caption}
\usepackage[font=small,labelfont=normalfont,labelformat=simple]{subcaption}
\graphicspath{{figs/}}

\newtheorem{theorem}{Theorem}

\newtheorem{lemma}[theorem]{Lemma}
\newtheorem{claim}[theorem]{Claim}

\newtheorem{problem}{Problem}

\author
{
Micha Christoph 
}
\thanks{Department of Computer Science, Institute of Theoretical Computer Science, ETH Z\"{u}rich, Switzerland.  \texttt{\{micha.christoph,kalina.petrova,raphaelmario.steiner\}@inf.ethz.ch}. Research of M.C. and R.S. funded by SNSF Ambizione grant No. 216071. Research of K.P. funded by SNSF grant No.  CRSII5 173721.}

\author
{
Kalina Petrova
}

\author
{
Raphael Steiner 
}

\date{\today}

\title{A note on digraph splitting}

\begin{document}
\maketitle

\begin{abstract}
A tantalizing open problem, posed independently by Stiebitz in 1995 and by Alon in 2006, asks whether for every pair of integers $s,t \ge 1$ there exists a finite number $F(s,t)$ such that the vertex set of every digraph of minimum out-degree at least $F(s,t)$ can be partitioned into non-empty parts $A$ and $B$ such that the subdigraphs induced on $A$ and $B$ have minimum out-degree at least $s$ and $t$, respectively. 

In this short note, we prove that if $F(2,2)$ exists, then all the numbers $F(s,t)$ with $s,t\ge 1$ exist and satisfy $F(s,t)=\Theta(s+t)$. 
In consequence, the problem of Alon and Stiebitz reduces to the case $s=t=2$. 
Moreover, the numbers $F(s,t)$ with $s,t \ge 2$ either all exist and grow linearly, or all of them do not exist. 
\end{abstract}

\section{Introduction}
A well-researched area in modern graph theory is that of \emph{graph splitting}. It is concerned with problems in which the vertex set of a given graph is to be split into a given number of parts while meeting specific degree conditions within or between the parts. One of the first instances of such a result is a classical theorem by Lov\'{a}sz~\cite{lovasz} from 1966, stating that for all numbers $s, t \in \mathbb{N}$, every graph $G$ of maximum degree $\Delta(G)\le s+t+1$ admits a partition of its vertex set into sets $A$ and $B$ such that $\Delta(G[A])\le s$ and $\Delta(G[B])\le t$. In the opposite direction, looking for splittings that preserve a given minimum degree, Thomassen~\cite{thomassen1} proved in 1983 that for all integers $s,t \ge 1$ there exists some $f(s,t) \in \mathbb{N}$ such that every graph $G$ of minimum degree $\delta(G)\ge f(s,t)$ admits a partition of its vertex set into non-empty sets $A$ and $B$ such that $\delta(G[A])\ge s$ and $\delta(G[B])\ge t$. He also conjectured that the function $f(s,t)$ can be taken to be $s+t+1$, which is best-possible as can be seen by considering complete graphs. In 1996, Stiebitz~\cite{stiebitz} proved Thomassen's conjecture. 




The perhaps most natural way of extending the above problems to directed graphs is to consider the out-degrees of vertices in a directed graph instead of their total degrees. Alon~\cite{alon} has written a short survey about the arising problems in 2006.
Maybe surprisingly, most of them turn out to be either false or much harder than their undirected cousins. In the following, we briefly summarize what is known.

\paragraph{\textbf{Maximum out-degree.}} The natural analogue of Lov\'{a}sz's theorem for directed graphs would state that for all $s,t \ge 1$, every directed graph $D$ of maximum out-degree $\Delta^+(D)\le s+t+1$ admits a partition $A, B$ of its vertex set such that $\Delta^+(D[A])\le s, \Delta^+(D[B])\le t$. However, this turns out to be completely false --- in 1983, Thomassen~\cite{thomassen3} constructed a sequence $(D_k)_{k=1}^{\infty}$ of digraphs such that $\Delta^+(D_k)=k$ for every $k$ and in every partition $A, B$ of $V(D_k)$, we either have $\Delta^+(D_k[A])=k$ or $\Delta^+(D_k[B])=k$. In other words, no matter how we split $D_k$ into two parts, the maximum out-degree of one of the two parts will not be reduced. However, the situation changes when allowing more than $2$ parts in the partition of the vertex set: Alon~\cite{alon} was the first to prove that every directed graph of maximum out-degree at most $\Delta$ can be split into three parts $A, B, C$ such that the maximum out-degree in each part is bounded by $\frac{2}{3}\Delta$. More generally, Alon's proof yields that for every digraph $D$ the vertex set of $D$ can be partitioned into three sets $A,B,C$ such that for every vertex $v$, at most $\frac{2}{3}d^+(v)$ of its out-neighbors lie in the same part of the partition as $v$. This statement has been independently reproved and strengthened several times, see~\cite{anholcer, girao, knox}. A popular conjecture due to Kreutzer, Oum, Seymour, van der Zypen and Wood~\cite{kreutzer} from 2017, known as the \emph{Majority Coloring Conjecture}, states that the constant $\frac{2}{3}$ in the above result can be improved to $\frac{1}{2}$, which would be best-possible. While this remains widely open, some special cases have been solved, such as tournaments and random graphs, see~\cite{anastos1, anastos2, girao}.

\paragraph{\textbf{Minimum out-degree.}} The natural analogue of Stiebitz's theorem for directed graphs would state that for all $s,t \ge 1$, every directed graph $D$ of minimum out-degree $\delta^+(D) \ge s+t+1$ admits a partition of its vertex set into non-empty sets $A, B$ such that $\delta^+(D[A]) \ge s$ and $\delta^+(D[B]) \ge t$. This statement, too, turns out to be quite false. Namely, Alon~\cite{alon0} proved in 1984 that for every integer $k \ge 1$ and every prime number $p>k^2\cdot 2^{2k-2}$ with $p\equiv 3 \text{ }(\text{mod }4)$, there exists a digraph $D$ of order $p$ such that $\delta^+(D)=\frac{p-1}{2}$ and such that for every non-empty $X\subseteq V(D)$ we have that $\delta^+(D[X])<\frac{k}{2}$ if $|X| \le k$ and $\delta^+(D[X])<\frac{p-1}{2}-k$ if $|X| \le p-k$. Setting $s=\frac{k}{2}$ and $t=\frac{p-1}{2}-k$, we can see that $\delta^+(D)>s+t+1$ but in every partition of $V(D)$ into non-empty sets $A, B$, we either have $|A|\le k$ and thus $\delta^+(D[A]) <s$, or $|B|\le p-k$ and thus $\delta^+(D[B])<t$. More recently, the third author~\cite{steiner} answered a question of Alon~\cite{alon} by proving that for arbitrarily large values of $s=t$, there exists a digraph $D$ with $\delta^+(D)>2s+(1+o(1))\log_3(s)>s+t+1$ such that for every non-empty $X\subseteq V(D)$ with $|X| \le \frac{|V(D)|}{2}$, we have $\delta^+(D[X])<s$. Similarly as above this implies a negative answer to the direct extension of Stiebitz's theorem when $s=t$.

The main open problem in this area is the intriguing question asked independently by Stiebitz in 1995 and Alon~\cite{alon} in 2006 whether the qualitative version of Stiebitz's theorem extends to directed graphs. We also refer to the open problem garden entry~\cite{splittingdigraphs}.

\begin{problem}\label{prob:main}
Does there exist, for all $s,t \ge 1$, a number $F(s,t) \in \mathbb{N}$ such that every digraph $D$ with $\delta^+(D) \ge F(s,t)$ has a partition $V(D)=A\sqcup B$ such that $\delta^+(D[A]) \ge s$ and $\delta^+(D[B])\ge t$?
\end{problem}
In the rest of this paper, we write $F(s,t)$ for the smallest possible integer satisfying the statement in Problem~\ref{prob:main} if it does exist, and set $F(s,t)=\infty$ otherwise. 

So far, $F(s,t)<\infty$ is only known in the case $s=t=1$, in which it is equivalent to the statement that every digraph of large enough minimum out-degree contains two disjoint directed cycles, see~\cite{thomassen2,alon2,bucic} for proofs and extensions of this statement. However, already whether $F(2,2)<\infty$ or even $F(1,2)<\infty$ remain open problems. The only other known results on Problem~\ref{thm:main} are for restricted classes of digraphs, for instance Alon et al.~\cite{alon3} and Yang et al.~\cite{yang} gave positive answers for tournaments and digraphs with balanced out- and in-degrees.

\paragraph{\textbf{Our result.}} As the main contribution of this paper, we show that in order to solve Problem~\ref{prob:main} in full generality, it suffices to decide whether $F(2,2)<\infty$. Moreover, under the assumption of $F(2,2)<\infty$, we settle the question of the asymptotic growth of $F(s,t)$ by showing that it is within a constant factor of the trivial lower bound $s+t+1$ for all values $s,t \ge 1$. We also obtain similar results for the values $F(s,1)$ if we assume that $F(2,1)<\infty$.
\begin{theorem}\label{thm:main}
\ 
\begin{enumerate}[label=(\arabic*)]
    \item\label{main_2} If $F(2,2)<\infty$, then $F(s,t)=\Theta(s+t)$ for $s,t \ge 1$. More precisely, we have 
$$F(s,t)< \frac{e^2}{3}\cdot F(2,2)^6\cdot \max\{s,t\}.$$ 
    \item\label{main_1} If $F(2,1)<\infty$, then $F(s,1)=\Theta(s)$ for $s\ge 1$. More precisely, we have 
$$F(s,1)< \frac{e^2}{3}\cdot F(2,1)^6\cdot s.$$ 
 \end{enumerate}
\end{theorem}

\section{Proof of Theorem~\ref{thm:main}}

Towards proving Theorem~\ref{thm:main}, we start with the following useful probabilistic lemma.
\begin{lemma}
\label{lemma:bipartite_graph}
For every integer $k \ge 3$ there exists a constant $\varepsilon= \frac{3}{e^2k^3}$ such that for every $n \in \mathbb{N}$ there exists a bipartite graph $G$ on $2n$ vertices with bipartition $V(G)=S \sqcup T$, such that $|S|=|T|=n$ and all of the following hold.
\begin{enumerate}[label=(\roman*)]
    \item\label{1} Every vertex in $S$ has degree exactly $k$.
    \item\label{2} For every non-empty $X\subseteq S$ with $|X| \le \varepsilon n$ and every $Y\subseteq T$ such that $|N_G(x)\cap Y| \ge 3$ for all $x \in X$, we have that $|Y|>|X|$.
\end{enumerate}
\end{lemma}
\begin{proof}
 Let $G$ be the random graph on vertex set $S\sqcup T$ obtained by choosing independently for each vertex $v\in S$ uniformly at random $k$ neighbors in $T$. Then, \ref{1} follows immediately, so let us show that \ref{2} also holds with positive probability. Given $Y\subseteq T$ and a set $U\subseteq T$ of three distinct vertices chosen uniformly at random, we get
    $$
        \Pr[U\subseteq Y]=\frac{|Y|(|Y|-1)(|Y|-2)}{n(n-1)(n-2)}\leq \left(\frac{|Y|}{n}\right)^3.
    $$
    It follows by a union bound that given a vertex $v\in S$ and a set $Y\subseteq T$,  
    $$
        \Pr[|N_G(v)\cap Y| \ge 3] \leq\binom{k}{3}\left(\frac{|Y|}{n}\right)^3\leq \frac{1}{6}\left(\frac{k|Y|}{n}\right)^3.
    $$
    Thus, for any non-empty sets $X\subseteq S$ and $Y\subseteq T$ we have
    $$
    \Pr[\forall x\in X:\ |N_G(x)\cap Y| \ge 3] \leq \left(\frac{1}{6}\right)^{|X|} \left(\frac{k|Y|}{n}\right)^{3|X|}.
    $$
    There are at most $\binom{n}{|X|}\binom{n}{|Y|}$ choices of $X$ and $Y$. Note that we only have to consider $X\subseteq S$ and $Y\subseteq T$ with $|X|=|Y|$, since if \ref{2} is not satisfied then there exist equal sized $X$ and $Y$ contradicting it. By a union bound, it follows that the probability that there exist such $X\subseteq S$ and $Y\subseteq T$ not satisfying \ref{2} is at most
    $$
        \sum_{i = 1}^{\lfloor\varepsilon n\rfloor}\binom{n}{i}^2\left(\frac{1}{6}\right)^i\left(\frac{ki}{n}\right)^{3i}\leq \sum_{i = 1}^{\lfloor\varepsilon n\rfloor}\left(\frac{en}{i}\right)^{2i}\left(\frac{1}{6}\right)^i\left(\frac{ki}{n}\right)^{3i}$$ $$ = \sum_{i = 1}^{\lfloor\varepsilon n\rfloor}\left(\frac{e^2k^3i}{6n}\right)^{i} < \sum_{i = 1}^{\infty}\left(\frac{e^2k^3\varepsilon}{6}\right)^{i}=\frac{e^2k^3\varepsilon}{6-e^2k^3\varepsilon}=1.
    $$
    Thus, there exists a graph $G$ which satisfies both conditions.
\end{proof}

\begin{proof}[Proof of Theorem~\ref{thm:main}]
We prove both \ref{main_2} and \ref{main_1} simultaneously. Towards this, for $k\in \mathbb{N}$, let $b(k) = k$ correspond to the proof of \ref{main_2} and $b(k)=1$ to the proof of \ref{main_1}. 

We first show that $F(3,b(3)) \leq F(4,b(4)) \leq F(2,b(2))^2$. The first inequality follows since $F$ is non-decreasing. For the second one, let $D$ be a digraph with $\delta^+(D) \geq F(2,b(2))^2$ and $V(D) = \{v_1, \dots, v_N\}$. Let $D'$ be the digraph obtained from $D$ by taking 
$$V(D') = V(D) \cup \{v_{i,j}|v_i \in V(D), 1 \le j \le F(2,b(2))\}$$
and the following arcs. For each $v_{i,j}\in V(D')$, the arc $(v_i, v_{i,j})$ is in $A(D')$. Moreover, for each $i \in [N]$, if $N^+_D(v_i) = \{u_i^1, \dots, u_i^{k_i}\}$ with $k_i \geq F(2,b(2))^2$, then for each $j=1,\ldots,F(2,b(2))$, we have that all the arcs
$$(v_{i,j}, u_i^{(j-1)F(2,b(2)) + \ell}), \forall \ell \in \{1,\ldots,F(2,b(2))\}$$
are in $A(D')$. Intuitively, to obtain $D'$, we have split the neighbourhood of each vertex $v_i$ into $F(2,b(2))$ groups of size $F(2,b(2))$ each, and added a different intermediate vertex on the path from $v_i$ to each group of its neighbourhood in $D$.

\begin{claim}\label{2 to 4}
    Let $W'\subseteq V(D')$ be non-empty and $W:=W'\cap V(D)$. 
    \begin{enumerate}[label=(\alph*)]
        \item If $\delta^+(D'[W'])\ge 1$, then $W$ is non-empty and $\delta^+(D[W])\ge 1$.
        \item If $\delta^+(D'[W'])\geq 2$ then $\delta^+(D[W])\geq 4$.
    \end{enumerate}
\end{claim}
\begin{proof}
\noindent
\begin{enumerate}[label=(\alph*)]
    \item Suppose $\delta^+(D'[W'])\geq 1$. This condition implies that $D'[W']$ contains a directed cycle $C$. It is easy to see by definition that $D'-V(D)$ is an acyclic digraph, hence $C$ must meet at least one vertex in $V(D)$. This shows that $V(C)\cap V(D)$ and hence also $W=W'\cap V(D)$ is non-empty. 
    
    Now, consider any vertex $v_i\in W$. Then there exists a vertex $v$ in $W'\cap \{v_{i,j} | 1\le j \le F(2,b(2))\}$ which in turn has an out-neighbor in $W'\cap N^+_D(v_i)$. This latter vertex is then an out-neighbor of $v_i$ in $D[W]$. Since $v_i \in W$ was chosen arbitrarily, this shows that $\delta^+(D[W])\ge 1$.
    \item Similarly, suppose that $\delta^+(D'[W'])\geq 2$. Let $v_i\in W$ be an arbitrary vertex. Since $\text{deg}^+_{D'[W']}(v_i) \geq 2$, there exist at least two vertices in $W'\cap \{v_{i,j} | 1 \le j \le F(2,b(2))\}$. Since each of them has at least $2$ out-neighbors in $D'[W']$, at least $4$ vertices among $N^+_D(v_i)$ are in $W'$ and so also in $W$. Since $v_i \in W$ was chosen arbitrarily, this shows that $\delta^+(D[W])\ge 4$.
\end{enumerate}
\end{proof}
By construction, $\delta^+(D') \geq F(2,b(2))$, so there is a partition of $V(D')$ into non-empty sets $A'$ and $B'$ such that $\delta^+(D'[A'])\geq 2$ and $\delta^+(D'[B']) \geq b(2)$. By Claim~\ref{2 to 4}, we get that $(A:=A'\cap V(D),B:=B' \cap V(D))$ is a partition of $V(D)$ into non-empty sets with $\delta^+(D[A])\geq 4$ and $\delta^+(D[B]) \geq b(4)$.

In the case of \ref{main_2}, suppose w.l.o.g. that $s \geq t$. Now, we show 
$$F(s,b(s)) \leq \left\lfloor\frac{e^2}{3} \cdot F(3,b(3))^3 \cdot s\right\rfloor =: d,$$
from which the result follows since $F$ is non-decreasing and $F(3,b(3))^3\leq F(2,b(2))^6$. Let $D$ be a digraph with $\delta^+(D) \geq d$ and $V(D) = \{v_1, \dots, v_N\}$. Let us define the digraph $D''$ obtained from $D$ with vertex set $$V(D'') = V(D) \cup \bigcup_{\substack{1 \le i \le N,\\3\leq j \leq s-1}} U_{i,j},$$ where each $U_{i,j}$ is a disjoint set of $d$ new vertices. Furthermore, for each $i=1,\ldots,N$, we select an arbitrary subset of size $d$ of $N^+_D(v_i)$ and denote it by $U_{i,s}$. For each $i=1,\ldots,N$ and each $u \in U_{i,3}$, we have that the arc $(v_i,u)$ is in $A(D'')$. 
Additionally, let $G$ be the graph with bipartition $S$ and $T$ given by Lemma~\ref{lemma:bipartite_graph} applied with $n:=d$ and $k:= F(3,b(3))$. Thus, $\varepsilon=\frac{3}{e^2F(3,b(3))^3}$. Then for all $1 \le i \le N$ and $3\leq j \leq s-1$, we add a copy of $G$ between $U_{i,j}$ and $U_{i,j+1}$ to $D''$, identifying $U_{i,j}$ with $S$ and $U_{i,j+1}$ with $T$, and directing all the edges from $U_{i,j}$ to $U_{i,j+1}$. Those are all the arcs in $D''$. Thus, by Lemma~\ref{lemma:bipartite_graph}, for each $3\leq j \leq s-1$, every vertex in $U_{i,j}$ has out-degree exactly $F(3,b(3))$ in $D''$. Now consider any non-empty subset $X \subseteq U_{i,j}$ with $|X|\leq s-1$. Then we have $$|X|<s-\varepsilon=\varepsilon\left(\frac{e^2F(3,b(3))^3}{3}  s-1\right) <\varepsilon d=\varepsilon n.$$ Thus by the second item of Lemma~\ref{lemma:bipartite_graph}, for every $Y \subseteq U_{i,j+1}$ with $|N^+_{D''}(x) \cap Y| \geq 3$ for all $x \in X$, we have $|Y| > |X|$.  

\begin{claim}\label{3 to s}
    Let $W'\subseteq V(D'')$ be non-empty and $W:=W'\cap V(D)$.
    \begin{enumerate}[label=(\alph*)]
        \item If $\delta^+(D''[W'])\ge 1$, then $W$ is non-empty and $\delta^+(D[W])\ge 1$.
        \item If $\delta^+(D''[W'])\geq 3$ then $\delta^+(D[W])\geq s$.
    \end{enumerate}
\end{claim}
\begin{proof}
\noindent
    \begin{enumerate}[label=(\alph*)]
        \item Suppose $\delta^+(D''[W'])\geq 1$. Similarly to before, this implies that $D''[W']$ contains a directed cycle $C$. Again we can observe that the digraph $D''-V(D)$ is acyclic, and thus we must have that $V(C)\cap V(D) \subseteq W'\cap V(D)=W$ is non-empty. Furthermore, from $\delta^+(D''[W'])\ge 1$ we get that for every $v_i\in W$ there exists $W_{i,j} \subseteq U_{i,j} \cap W'$ with $|W_{i,j}| = 1$ for every $3\leq j\leq s$, implying that 
    $$|N^+_{D[W]}(v_i)| \geq |W_{i,s}| \geq 1.$$
        \item Suppose $\delta^+(D''[W'])\geq 3$ and consider some $v_i \in W$. We now show by induction on $j$ that for each $3\leq j\leq s$, there is a set $W_{i,j} \subseteq U_{i,j} \cap W'$ with $|W_{i,j}| = j$. For the base case $j=3$, since $\text{deg}^+_{D''[W']}(v_i) \geq 3$, there is some $W_{i,3} \subseteq U_{i,3} \cap W'$ with $|W_{i,3}| = 3$. Suppose we have shown that the statement holds for some $3\leq j \leq s-1$, we now show it holds for $j+1$. Let $X:=W_{i,j}$ and $Y := U_{i,j+1} \cap W'$. Since $W_{i,j} \subseteq W'$, it must hold that for each $x \in X$, we have $$|N^+_{D''[W']}(x)| = |N^+_{D''}(x) \cap Y| \geq 3.$$ Thus, as $|X|=|W_{i,j}| = j \leq s-1$, by the above-mentioned properties coming from Lemma~\ref{lemma:bipartite_graph} we have that $|Y| > |X|=|W_{i,j}| = j$. Taking $W_{i,j+1}$ to be an arbitrary subset of $Y$ of size $j+1$ finishes the induction step. In particular, this shows that 
    $$|N^+_{D[W]}(v_i)| \geq |U_{i,s} \cap W| = |U_{i,s} \cap W'| \geq |W_{i,s}| \geq s,$$
    as desired.
    \end{enumerate}
\end{proof}
Since $\delta^+(D'') \geq k=F(3,b(3))$ by construction, there is a partition of $V(D'')$ into non-empty sets $A'$ and $B'$ such that $\delta^+(D''[A'])\geq 3$ and $\delta^+(D''[B']) \geq b(3)$. It follows by Claim~\ref{3 to s} that $A:=A' \cap V(D)$ and $B:= B' \cap V(D)$ are non-empty and satisfy $\delta^+(D[A]) \ge s$ and $\delta^+(D[B])\ge b(s)$, as desired. This shows that $F(s,b(s))\le d$, concluding the proof.
\end{proof}

\end{document}